\documentclass[oneside,11pt]{amsart}
\usepackage[utf8]{inputenc}
\usepackage[T1]{fontenc}
\usepackage{amsmath, mathrsfs, bbm}
\usepackage{graphicx}
\usepackage[T1]{fontenc}

\usepackage[small]{caption}

\usepackage{tikz-cd}
\usetikzlibrary{arrows,chains,matrix,positioning,scopes}

\makeatletter
\tikzset{join/.code=\tikzset{after node path={%
\ifx\tikzchainprevious\pgfutil@empty\else(\tikzchainprevious)%
edge[every join]#1(\tikzchaincurrent)\fi}}}
\makeatother
\tikzset{>=stealth',every on chain/.append style={join},
         every join/.style={->}}
\tikzstyle{labeled}=[execute at begin node=$\scriptstyle,
   execute at end node=$]

\usepackage{amsfonts}
\usepackage{amssymb}
\usepackage{amsxtra}
\usepackage{ifthen}

\newcommand{\showcomments}{no}

\newsavebox{\commentbox}
{\ifthenelse{\equal{\showcomments}{yes}}%
{\footnotemark
    \begin{lrbox}{\commentbox}
    \begin{minipage}[t]{1.25in}\raggedright\sffamily\upshape\tiny
    \footnotemark[\arabic{footnote}]}
{\begin{lrbox}{\commentbox}}}%
{\ifthenelse{\equal{\showcomments}{yes}}%
{\end{minipage}\end{lrbox}\marginpar{\usebox{\commentbox}}}
{\end{lrbox}}}

\theoremstyle{plain}
\newtheorem{theorem}{Theorem}[section]

\newtheorem{lemma}[theorem]{Lemma}
\newtheorem{proposition}[theorem]{Proposition}

\theoremstyle{definition}

\newtheorem{definition}[theorem]{Definition}
\newtheorem{remark}[theorem]{Remark}

\newcommand{\R}{\mathbb{R}}

\newcommand{\A}{\mathcal{A}}

\renewcommand{\emptyset}{\varnothing}
\renewcommand{\setminus}{-}

\newcommand{\field}[1]{\mathbb{#1}}

\renewcommand{\P}{\field{P}}

\usepackage{combelow}

\begin{document}
\title[Connectedness of Bowditch Boundary of Dehn Fillings]{Connectedness of Bowditch Boundary of Dehn Fillings}

\author{Ashani Dasgupta}
\address{
Cheentan Research\\
Kolkata, WB 700029\\
India}
\email{ashani.dasgupta@cheenta.com}

\date{\today}

\begin{abstract}
We study Dehn fillings of relatively hyperbolic group pairs $(\Gamma, \P)$ and the persistence of connectedness of Bowditch boundary in sufficiently long Dehn fillings. We show that the restriction of peripheral subgroups to virtually polycyclic subgroups (as in \cite{DehnFillings}) is not needed.
\end{abstract}

\maketitle

\section{Introduction}

In \cite{GrovesManning08_DehnRelHyp} and \cite{Osin07_Peri}, Groves, Manning and Osin study group-theoretic Dehn fillings. In this paper we look into Dehn fillings of relatively hyperbolic group pairs in the spirit of \cite{GrovesManning08_DehnRelHyp}.

Our main result concerns the persistence of properties of relative oneendedness, nonsplitting over parabolic subgroups and connectedness properties of Bowditch boundary. These questions are studied with restriction on peripheral subgroups by Groves and Manning in \cite{DehnFillings}.

\begin{theorem}\cite[Theorem 1.6]{DehnFillings}
Let $G$ be a group which is hyperbolic relative to a finite collection $\P$ of subgroups, and suppose that all small subgroups of $G$ are finitely generated. Furthermore, suppose that $G$ admits no nontrivial elementary splittings. Then all sufficiently long $\mathcal{M}$-finite co-slender fillings $(G, \P) \rightarrow(\bar{G}, \bar{\P})$ have the property that $\bar{G}$ admits no nontrivial elementary splittings.
\end{theorem}

We study the relative case of this theorem and show that the assumption of small subgroups being finitely generated is no longer needed if the splittings (or non-splittings) are assumed to be relative to $\P$. This required certain observations in the proofs contained in \cite{DehnFillings} and an application of a relative version of Rips machine due to \cite{Guirardel2012SplittingsAA}. We prove the following theorem. 

\begin{theorem}
\label{thm: elementary splittings is preserved}
Let $(G, \P)$ be a relatively hyperbolic group pair that admits no elementary splitting relative to $\P$. Then all sufficiently long $\mathcal{M}$--finite fillings $(G, \P) \twoheadrightarrow (\bar{G}, \bar{\P})$ have the property that $\bar{G}$ admits no elementary splitting relative to $\bar{\P}$.
\end{theorem}

Next we study the boundary of sufficiently long Dehn fillings. Groves and Manning prove the following theorem with some assumptions on peripheral subgroups.

\begin{theorem}\cite[Theorem 1.9]{DehnFillings}
Suppose that $(G, \P)$ is relatively hyperbolic, with $\mathcal{P}$ consisting of virtually polycyclic groups. Suppose further that the Bowditch boundary $\partial(G, \P)$ is connected with no cut point. Then for all sufficiently long $\mathcal{M}$-finite fillings $(G, \P) \rightarrow$ $(\bar{G}, \bar{\P})$, the resulting boundary $\partial\left(\bar{G}, \bar{\P}^{\infty}\right)$ is connected and has no cut points.
\end{theorem}

We show that the assumption of peripheral subgroups being virtually polycyclic is not necessary. In this context we use JSJ theory (see \cite{GuirardelLevitt17_jsjDecom}) and a cut point theorem due to \cite{dasgupta} to study the boundary. We prove the following theorem.

\begin{theorem}
Suppose $(G, \P)$ is relatively hyperbolic. Suppose further that the Bowditch boundary $\partial (G, \P)$ is connected with no cut point.

Then for all sufficiently long $\mathcal{M}$--finite fillings $(G, \P) \twoheadrightarrow (\bar{G}, \bar{\P})$, the resulting boundary of $\partial (\bar{G}, \bar{\P}^{\infty} )$ is connected and has no cut points.
\end{theorem}

\textit{Acknowledgements.} We thank Christopher Hruska and Daniel Groves for helpful comments.

\section{Preliminaries}

In this section we put definitions and results used in this paper. For more details see \cite{DehnFillings} and \cite{dasgupta}. We begin with a definition for relatively hyperbolic group pairs. There are several equivalent definitions of relative hyperbolicity. The definition we use here is from \cite{MackeySisto20_quasiHypPlane}. It is equivalent to Bowditch \cite{Bowditch12_relHyp} and \cite{DehnFillings}. See \cite[Theorem 5.1]{Hruska10_relHypRelQuasiConvexity} for more details.

\begin{definition}
A geodesic space $X$ is \emph{$\delta$--hyperbolic} if each side of a geodesic triangle lies in the $\delta$--neighborhood of the union of the other two sides.
\end{definition}

\begin{definition}
A \emph{group pair} $(\Gamma,\mathbb{P})$ consists of a group $\Gamma$ together with a finite collection $\mathbb{P}$ of subgroups of $\Gamma$.
Given a group pair $(\Gamma,\mathbb{P})$, an action of $\Gamma$ on $X$ is \emph{relative to $\mathbb{P}$} if each member of $\mathbb{P}$ has a fixed point in $X$.
\end{definition}

\begin{definition}[Horoball]

Let $\Gamma$ be a connected graph with $V$ and $E$ denoting the set of vertices and edges of $\Gamma$, such that every edge has length $1$. Let $T = [0, 1]\times [1, \infty) \subset \mathbb{H}^2$ in the upper half plane model of the hyperbolic plane. Glue a copy of $T$ to each edge in $E$ along $[0, 1] \times \{ 1\} $ and identify the rays $\{v\} \times [1, \infty)$ for all $v \in V$. The quotient space with the natural path metric is defined as \textit{the} horoball $\mathcal{B}(\Gamma)$.
\end{definition}

\begin{definition}[Cusped space, relatively hyperbolic group pair, peripheral subgroups]
Let $G$ be a finitely generated group and $\P = \{P_1, \cdots, \P_k\}$ is a collection of proper finitely generated subgroups of $G$. Suppose $S$ be a finite generating set for $G$, so that $ S \cap P_i$ generates $P_i$ for each $i = 1, \cdots , k$. 

Let $\Gamma(G)= \Gamma(G, S)$ be the Cayley graph of $G$ with respect to $S$, with word metric $d_G$. Let $Y$ be the disjoint union of $\Gamma(G)$ and copies of $\mathcal{B}_{h P_i}$ of $\mathcal{B}(\Gamma(P_i, S \cap P_i)$ for each left coset $h P_i$ and each $P_i \in \P$. Let $X = X(G, \P) = Y / \sim $, where for each left coset $h P_i$  and each $g \in hP_i$ the equivalence relation $\sim$ identifies $g \in \Gamma(G, S)$ with $(g, 1) \in \mathcal{B}_{h P_i}$. We endow $X$ with the induced path metric $d$, which makes $(X, d)$ a proper geodesic metric space.

We say that $X$ is \textit{cusped space}. Furthermore $(G, \P)$ is a \textit{relatively hyperbolic group pair} if $(X, d)$ is Gromov hyperbolic. We call the members of $\P$ \textit{peripheral subgroups}.
\end{definition}

\begin{remark}
In this paper we consider only finitely generated relatively hyperbolic groups. We do not assume that the elements of $\P$ are infinite. We denote the collection of infinite members $\P$ by $\P^{\infty}$ and collection of nonhyperbolic subgroups as $\P^{red}$.
\end{remark}

\begin{definition}[Elementary subgroup]

Any subgroup of conjugates of the members $\P$ is called \textit{parabolic subgroup}. A subgroup $E < G$ is \textit{elementary} if it is either finite, two ended or parabolic.
\end{definition}

\begin{definition}[Elementary splitting] Suppose a relatively hyperbolic group $G$ splits as a graph of groups such that the edge groups are elementary subgroups of $G$. We call such splittings as \textit{elementary splittings}.
\end{definition}

Let $G$ be any group, and let $\mathbb{A}$ be a family of subgroups closed under conjugation.
Suppose $G$ acts on a simplicial tree $T$ without inversions and with no proper invariant subtree.
Then $T$ is an \emph{$\mathbb{A}$--tree} if each edge stabilizer is a member of $\mathbb{A}$.
Suppose $(G,\mathbb{P})$ is a group pair. An \emph{$(\mathbb{A},\P)$--tree} is an $\mathbb{A}$--tree $T$ such that each member of $\P$ has a fixed vertex in $T$.
Two $(\mathbb{A},\P)$--trees are \emph{equivalent} if there are $\Gamma$--equivariant maps $T\to T'$ and also $T'\to T$. If $T$ is not a point, then we say that $G$ \textit{splits over} $\A$ \textit{relative to} $\P$.

\begin{definition}[Peripheral Splitting]

Suppose that $(G, \P)$ is a relatively hyperbolic group pair. A \textit{peripheral splitting} of $(G, \P)$ is a bipartite graph of groups with fundamental group $G$ where the vertex groups of one color are precisely the conjugates of peripheral subgroups $\P$.
\end{definition}

\begin{definition}

An action of a group $G$ on a tree $T$ in $(k, C)$-\textit{acylindrical} if the stabilizer of any segment of length at least $k+1$ has cardinality at most $C$.

\end{definition}

\begin{definition}

A collection of subgroups $\P$ os a group $G$ is $C$-\textit{almost malnormal} if there is a constant $C$ so that $$|P_1 \cap g P_2 g^{-1}) | > C, \quad \textrm{for} \quad g \in G, P_1, P_2 \in \P$$ implies $P_1 = P_2$ and $g \in P_1$.
\end{definition}

\begin{lemma} \cite[Lemma 3.3]{DehnFillings}
Suppose $(G, \P)$ is a relatively hyperbolic group pair. Then $\P$ is $C$-almost malnormal for some $C$.
\end{lemma}

\begin{definition}
Suppose that $(G, \P)$ is a relatively hyperbolic group pair, and let $\mathcal{F}$ be the set of subgroups of $G$ which are either finite nonparabolic or contained in the intersection of two distinct maximal parabolic subgroups. Define $C(G, \P) = max \{ |F| \colon  F \in \mathcal{F} \}$.
\end{definition}

\begin{lemma}\cite[Lemma 4.4]{DehnFillings}
For $(G, \P)$ relatively hyperbolic, $$C(G, \P) < \infty$$.
\end{lemma}

\subsection{Dehn Fillings}

\begin{definition}
Suppose that $G$ is a group and $\P = \{ P_1, \cdots , P_n\}$ is a collection of subgroups. A \textit{Dehn filling} (or just \textit{filling}) of $(G, \P)$ is a quotient map: $\phi: G \to G/K$, where $K$ is the normal closure in $G$ of some collection $K_I \trianglelefteq P_i$. We write $$ G/K = G(K_1,\cdots, K_n)$$ for this quotient. The subgroups $K_1, \cdots , K_n$ are called the \textit{filling kernels}. We also write $\phi : (G, \P) \to (\bar{G},\bar{\P})$ where $\bar{\P}$ is the collection of images of all the $P\in \P$.

We say that a property holds \textit{for all sufficiently long fillings} of $(G, \P)$ if there is finite $\mathcal{B} \in G \setminus \{1\}$ so that whenever $K_i \cap \mathcal{B} = \emptyset $ for all $i$, the group $G/K$ has the property.
\end{definition}

\begin{proposition} \cite[Proposition 3.4]{DehnFillings} \label{prop: fillingsMalnormal}
If $(G, \P)$ is relatively hyperbolic, then $\P$ is $C$-almost malnormal, then for all sufficiently long fillings $(\bar{G},\bar{\P})$ of
$(G, \P)$, the collection $\bar{\P}$ is $C$-almost malnormal.
\end{proposition}

\begin{definition}
A group $H$ is \textit{small} if $H$ has no subgroup isomorphic to a non-abelian free group.
A group $H$ is \textit{slender} if every subgroup of $H$ is finitely generated.
\end{definition}

\begin{definition}
Let $(G, \P)$ be a group pair and let $\mathcal{M}$ be the class of all finitely generated groups with more than one end. We say that a filling of $(G, \P)$ is $\mathcal{M}$-\textit{finite} if for all $P \in \mathcal{M} \cap \P$, the associate filling kernel $K \trianglelefteq P$ has finite index in $P$. We say that a filling of $(G, \P)$ is \textit{co-slender} if for all $P\in \P$ with associated filling $K$, the group $P/K$ is slender.
\end{definition}

\begin{remark}
Suppose $\phi : (G, \P) \to (\bar{G},\bar{\P})$ is a sufficiently long Dehn filling of a relatively hyperbolic group pair. Note that $(\bar{G},\bar{\P})$ is also relatively hyperbolic by \cite[Theorem 2.17]{DehnFillings}. Moreover $(\bar{G},\bar{\P}^{\infty})$ and $(\bar{G},\bar{\P}^{red})$ are also relatively hyperbolic group pairs though their Bowditch boundary may be different (see \cite[Section 7.2]{DehnFillings}).
\end{remark}

\section{Relative Elementary splittings}

In this section we study the elementary splittings of a relatively hyperbolic group pair $(G,\P)$ relative to $\P$. We upgrade such splittings to $(2, C)$-acylindrical splittings relative to $\P$. All of the results in this section can be easily deduced from the proofs in Section 3 and 4 in \cite{DehnFillings}. The slenderness related hypothesis becomes unnecessary as we consider elementary splittings relative to $\P$. We put the relevant statements for the sake of completion.

We begin by studying splittings over parabolic subgroups. The following lemma follows directly from the proof of \cite[Lemma 3.6]{DehnFillings}. In fact, as we consider relative splittings, the hypothesis does not require the slenderness of parabolic subgroups.

\begin{lemma}
\label{lem: parabolic to 2C acylindrical upgrade}
Suppose $(G, \P)$ is relatively hyperbolic where $\P$ is a $C$-almost malnormal collection of subgroups. If $G$ admits a nontrivial splitting over a parabolic group relative to $\P$, then $G$ admits a $(2, C)$--acylindrical splitting over a parabolic subgroup relative to $\P$.
\end{lemma}

\begin{proof}
Let $H$ be a parabolic subgroup in $G$ and suppose $G$ splits over $H$ relative to $\P$. Without loss of generality, $H < P$ for some $P \in \P$. Since $\P$ is $C$-almost malnormal, if $|H| \leq C$ then we are done. Otherwise suppose that $|H| > C$.

Let $T$ be the Bass-Serre tree for this splitting. $P$ fixes some point $p \in T$. 

If $P$ fixes some edge $e \in T$, then each stabilizer is $g^{-1}P{g}$ for some $g \in G$ and $P =H$. Since $\P$ is $C$-almost malnormal, any segment of length $2$ or more has stabilizer of size $C$ or less. Hence the action of $G$ is $(1, C)$-acylindrical. 

If $P$ does not fix a vertex and not an edge in $T$, then stabilizer of any segment of length 3 or more is contained in a pair of conjugates of $P$. Hence the size of the stabilizer is $C$ or less. In this case the action of $G$ on $T$ is $(2, C)$-acylindrical.
\end{proof}

The following proposition follows immediately from \cite[Proposition 3.7]{DehnFillings} and does not require co-slenderness hypothesis Lemma~\ref{lem: parabolic to 2C acylindrical upgrade} does not need slenderness hypothesis. We include a proof for the sake of completion.
    
\begin{proposition}
\label{prop: parabolic to 2C acylindrical upgrade in fillings}
Suppose $(G, \P)$ is a relatively hyperbolic group pair. For all sufficiently long fillings $(G, \P) \to (\bar{G}, \bar{P})$, if $\bar{G}$ admits a nontrivial splitting over a parabolic subgroup relative to $\P$ then $\bar{G}$ admits a non trivial $(2, C)$--acylindrical splitting over a parabolic subgroup relative to $\P$.
\end{proposition}

\begin{proof}
If $(\bar{G}, \bar{\P})$ is a sufficiently long filling, then by Proposition~\ref{prop: fillingsMalnormal}, the collection $\bar{\P}$ is $C$-malnormal. Then by Lemma~\ref{lem: parabolic to 2C acylindrical upgrade}, we conclude that $\bar{G}$ admits a non trivial $(2, C)$-acylindrical splitting over a parabolic group relative to $\P$.
\end{proof}

\begin{lemma}\cite[Corollary 4.5]{DehnFillings}
\label{lem: C Persists}
Let $(G,\P)$ be relatively hyperbolic group pair. For all sufficiently long fillings $(\bar{G}, \bar{\P})$, we have $C(\bar{G}, \bar{\P}) \leq C(G,\P)$.
\end{lemma}

Next we examine the splittings over finite subgroups. Notice that the following lemma works without co-slenderness hypothesis as Proposition~\ref{prop: parabolic to 2C acylindrical upgrade in fillings} does not require co-slenderness hypothesis. 
    
\begin{lemma}
\label{cor: relative finite splitting}

Suppose $(G, \P)$ is relatively hyperbolic. For all sufficiently long fillings $(G, \P) \to (\bar{G}, \bar{P})$, if $\bar{G}$ admits a non trivial splitting over a finite group relative to $\P$, then $\bar{G}$ admits a nontrivial $(2, C)$-acylindrical splitting over a finite or parabolic group relative to $\P$, where $C = C (G, \P)$.
\end{lemma}

\begin{proof}
Suppose $\bar{G}$ admits a splitting over a finite subgroup $K$ relative to $\P$. Then $K$ is a either parabolic or non-parabolic.

If $K$ is non-parabolic, then by Lemma~\ref{lem: C Persists}, if the filling is sufficiently long, then $|K|\leq C$. Therefore the Bass-Serre tree corresponding to the splitting is $(0, C)$-acylindrical.

If $K$ is parabolic then by Proposition~\ref{prop: parabolic to 2C acylindrical upgrade in fillings} we are done.
\end{proof}
    
Next we examine the case of splittings over two-ended subgroups relative to $\P$. Notice that the relative version of \cite[Lemma 4.9]{DehnFillings} does not require the co-slenderness hypothesis as Lemma~\ref{lem: parabolic to 2C acylindrical upgrade} does not need slenderness hypothesis. Hence we have the following lemma from the proof of \cite[Lemma 4.9]{DehnFillings}.
    
\begin{lemma}
\label{lem: two ended non parabolic to acylindrical}
Let $(G, \P)$ be relatively hyperbolic and let $C = 2C(G,\P)$. If $G$ admits a nontrivial splitting over a two-ended non-parabolic subgroup relative to $\P$, then $G$ admits a nontrivial $(2, C)$--acylindrical splitting over an elementary subgroup relative to $\P$.
\end{lemma}

\begin{proof}
Suppose $G$ splits over non-parabolic two ended subgroup $H$ relative to $\P$ and $\hat{H}$ is the maximal two ended subgroup containing $\hat{H}$. Extend the peripheral structure to $\P' = \P \cup \{ \hat{H}\}$. Note that $(G, \P')$ is also a relatively hyperbolic group pair and $C(G, \P') \leq C$ by \cite[Lemma 4.8]{DehnFillings}. Therefore by Lemma~\ref{lem: parabolic to 2C acylindrical upgrade}, note that $G$ admits a $(2, C)$-acylindrical splitting over $\P'$-parabolic subgroup $H'$. If $H'$ is conjugate into a member of $\P$ then we are done. Otherwise $H'$ is cojugate to a subgroup of $H$ hence it must be finite or two-ended. This provides us with the necessary elementary splitting. 
\end{proof}
    
Finally we have the relative version of \cite[Proposition 4.10]{DehnFillings}. It does not require the co-slenderness hypothesis as Lemma~\ref{lem: two ended non parabolic to acylindrical} does not need co-slenderness hypothesis. 

\begin{proposition}
\label{prop: twoended to 2C acylindrical upgrade in fillings}
Suppose $(G, \P)$ be relatively hyperbolic and let $C = 2C(G,\P)$. For all sufficiently long fillings $(G, \P) \to (\bar{G}, \bar{\P})$ if $\bar{G}$ admits a nontrivial splitting over a two-ended non-parabolic subgroup relative to $\P$, then $\bar{G}$ admits a nontrivial $(2, C)$--acylindrical splitting over an elementary subgroup relative to $\P$.
\end{proposition}

\begin{proof}
$C(\bar{G}, \bar{\P}) \leq C(G,\P)$ for sufficiently long fillings by Lemma~\ref{lem: C Persists}. Then by Lemma~\ref{lem: two ended non parabolic to acylindrical}, we are done.
\end{proof}

\section{Action on $\mathbb{R}$--tree}

In this section we construct an $\mathbb{R}$--tree starting with a sequence of fillings of a relatively hyperbolic group pair $(G, \P)$. We will use this tree in subsequent sections along with a structure theorem to prove the main theorems. We first prove a lemma that upgrades a sequence of elementary splittings to $(2, C)$-acylindrical elementary splittings. 

\begin{lemma}
\label{lem: elementary to 2C acylindrical upgrade in fillings}
Suppose that $(G, \P)$ is a relatively hyperbolic group pair such that there is a stably faithful sequence of $\mathcal{M}$--finite fillings $\phi_i: (G, \P) \to (\bar{G}_i , \bar{\P}_i)$, so that each $\bar{G}_i$ admits a nontrivial elementary splitting relative to $\bar{\P}_i$. 

Let $C= 2C(G,\P)$. Then there is a stably faithful sequence of $\mathcal{M}$--finite filings $\eta_i: (G, \P) \to (\bar{G}_i , \bar{\P}_i)$, so that each $\bar{G}_i$ admits a nontrivial $(2, C)$-acylindrical elementary splitting relative to $\bar{\P}_i$.
\end{lemma}

\begin{proof}
By Corollary \ref{cor: relative finite splitting}, Proposition \ref{prop: parabolic to 2C acylindrical upgrade in fillings} and Proposition \ref{prop: twoended to 2C acylindrical upgrade in fillings}, for sufficiently large $i$, the conclusion immediately follows.
\end{proof}

Suppose we have a relatively hyperbolic group pair $(G, \P)$ and a stably faithful sequence of $\mathcal{M}$--finite filings $\eta_i: (G, \P) \to (\bar{G}_i , \bar{\P}_i)$, so that each $\bar{G}_i$ admits a nontrivial $(2, C)$-acylindrical elementary splitting relative to $\bar{\P}_i$. Then each $\bar{G}_i$ acts $(2, C)$-acylindrically on the Bass Serre tree $T_i$ such that each member of $\bar{\P}_i$ is conjugated into the vertex groups of $T_i$. The argument in \cite[Section 5]{DehnFillings} goes through and we construct a \textit{limiting tree} $T_{\infty}$ on which $G$ acts relative to $\P$. 

\begin{lemma}\cite[Section 5]{DehnFillings}
\label{lem: RtreeActionByG}
The action of $G$ on $T_{\infty}$ has no global fixed point and relative to $\P$ with elementary arc stabilizers.
\end{lemma}

Notice that none of the results used in the following proof uses the fact that small subgroups of $G$ are finitely generated. 

\begin{proof}
$G$ acts on $T_{\infty}$ without global fixed points by \cite[Lemma 5.4]{DehnFillings}. Using the argument in the proof of \cite[Lemma 5.5]{DehnFillings}, we conclude that the action is relative to $\P$. By \cite[Corollary 5.8]{DehnFillings}, the arc stabilizers are small and therefore by \cite[Corollary 5.11]{DehnFillings}, the arc stabilizers are elementary. 
\end{proof}

The fundamental structure theorem for stable actions of finitely presented groups on $\R$--trees is the splitting theorem of Bestvina--Feighn \cite{BestvinaFeighn_RipsMachine}.
A relatively hyperbolic analogue of the Bestvina--Feighn structure theorem due to Guirardel--Levitt \cite{GuirardelLevitt_Splittings} gives the following theorem.

\begin{theorem}\cite[Corollary 9.10]{Guirardel2012SplittingsAA}
\label{cor: relativerips thm}
Let $(G, \P)$ be a relatively hyperbolic group pair. If $G$ acts non-trivially on an $\mathbb{R}$-tree $T$ relative to $\P$ with elementary
arc stabilizers, then $G$ splits over an elementary subgroup relative to $\P$ \end{theorem}

\section{Main Results}
The first part of the proof of the following theorem is essentially available in \cite{DehnFillings}. We feed the Dehn fillings into Theorem \ref{cor: relativerips thm} to obtain the conclusion. We state the statement and sketch the structure of the argument for the sake of completion.
\begin{theorem}
\label{thm: elementary splittings is preserved}
Let $(G, \P)$ be a relatively hyperbolic group pair that admits no elementary splitting relative to $\P$. Then all sufficiently long $\mathcal{M}$--finite fillings $(G, \P) \twoheadrightarrow (\bar{G}, \bar{\P})$ have the property that $\bar{G}$ admits no elementary splitting relative to $\bar{\P}$.
\end{theorem}
% \begin{enumerate}
%     \item Is the action of $G$ on $T_{\infty}$ relative to $\P$?
%     \item Does the proof of \cite[Corollary 5.8]{DehnFillings} depend on the hypothesis $\colon$ `all small subgroups are finitely generated'?
% \end{enumerate}

\begin{proof}
Suppose $(G, \P)$ is a counterexample to this theorem. Then there exists a non-trivial action of $G$ on an $\mathbb{R}$--tree $T_{\infty}$ relative to $\P$ with elementary arc stabilizers by Lemma~\ref{lem: elementary to 2C acylindrical upgrade in fillings} and Lemma~\ref{lem: RtreeActionByG}. Then by Theorem~\ref{cor: relativerips thm} we conclude that $G$ admits an elementary splitting relative to $\P$ which is a contradiction.
\end{proof}

\begin{theorem}
\label{thm: parabolic splittings is preserved}

Let $(G, \P)$ be relatively hyperbolic, relatively one-ended and admits no proper peripheral splittings. Then all sufficiently long $\mathcal{M}$--finite fillings $(G, \P) \twoheadrightarrow (\bar{G}, \bar{\P})$ have the property that $\bar{G}$ is one-ended relative to $\bar{\P}^{\infty}$ and admits no splitting over parabolic subgroups relative to $\bar{\P}^{\infty}$.

\end{theorem}

\begin{proof}

Let $G$ be a counterexample to this theorem. Then $\bar{G}$ is not one ended relative to $\bar{\P}^{\infty}$. Hence $\bar{G}$ admits a splitting over a finite subgroup $F$ relative to $\bar{\P}^{\infty}$. We have two cases $\colon$ either $F$ is parabolic or $F$ is nonparabolic subgroup of $\bar{G}$.

Suppose $F$ is parabolic. As $\bar{G}$ admits a splitting over a parabolic subgroup relative to $\bar{\P}^{\infty}$, therefore $(\bar{G},\bar{\P}^{\infty})$ admits a peripheral splitting in which all members of $\bar{\P}$ are elliptic. Consider the induced action of $G$ on the Bass--Serre tree $T$ of this peripheral splitting $(\bar{G},\bar{\P})$. Clearly all the members of $\P$ act elliptically and all the edge stabilizers are parabolic. This implies a proper peripheral splitting of $(G, \P)$ contradicting the given hypothesis.

Now suppose $F$ is non-parabolic. Let $\mathcal{A}$ be the collection of non-parabolic finite groups. The order of the members in $\mathcal{A}$ is bounded by \cite[Lemma 4.3]{DehnFillings}. Hence by \cite[Section 3.3]{GuirardelLevitt17_jsjDecom}, there is a JSJ splitting $T$ over $\mathcal{A}$ relative to $\bar{\P}^{\infty}$. Again consider the induced action of $G$ on $T$. Since $F$ is non-parabolic, hence it is isomorphic to a finite subgroup $F'$ of $G$ by \cite[Theorem 4.1]{DehnFillings}. Clearly the action of $G$ on $T$ is relative to $\P$ with a finite group $F'$ as an edge stabilizer of $T$ contradicting the relative one endedness of $G$.
\end{proof}

\begin{theorem}\cite[Theorem 1.1]{dasgupta}
\label{thm: cutpoint implies peripheral splitting}
Suppose $(G, \P)$ is relatively hyperbolic. Suppose further that the Bowditch boundary $\partial (G, \P)$ is connected. Then $\partial (G, \P)$ has a cut point if and only if $(G, \P)$ has a non trivial peripheral splitting.

\end{theorem}

% \begin{proof}
% Suppose $(G, \P)$ has a nontrivial peripheral splitting. Then by \cite[Theorem 1.2]{bowditchPeripheralSplitting} we conclude $\partial (G, \P)$ has a cut point.

% Conversely suppose $\partial (G, \P)$ has a cut point. By \cite{dasgupta}, that cut point must be a parabolic fixed point. Hence by \cite[Theorem 0.2]{Bowditch.boundary.of.geom} we conclude that $(G, \P)$ admits a proper peripheral splitting.
% \end{proof}

\begin{theorem}
Suppose $(G, \P)$ is relatively hyperbolic. Suppose further that the Bowditch boundary $\partial (G, \P)$ is connected with no cut point.

Then for all sufficiently long $\mathcal{M}$--finite fillings $(G, \P) \twoheadrightarrow (\bar{G}, \bar{\P})$, the resulting boundary of $\partial (\bar{G}, \bar{\P}^{\infty} )$ is connected and has no cut points.
\end{theorem}

\begin{proof}
Since there is not cut point in the boundary, hence by Theorem \ref{thm: cutpoint implies peripheral splitting}, we know that $(G, \P)$ has no proper peripheral splitting.

Then by Theorem \ref{thm: parabolic splittings is preserved}, for sufficiently long $\mathcal{M}$--finite fillings $(G, \P) \twoheadrightarrow (\bar{G}, \bar{\P})$, we know that $(\bar{G}, \bar{\P}^{\infty})$ is relatively one-ended and has no proper peripheral splittings. Then by \cite[Theorem 10.1]{Bowditch12_relHyp}, we conclude that $\partial (\bar{G}, \bar{\P}^{\infty} )$ is connected and by \cite[Theorem 1.1]{dasgupta}, it has no cut points.
\end{proof}

\bibliographystyle{alpha}
\bibliography{main}

\end{document}